\documentclass[12pt]{article}

\usepackage[english]{babel}
\usepackage{upref,amsfonts,amsxtra,latexsym,color}
\usepackage{amssymb,amsmath,amsthm,a4wide,epsf,mathrsfs,verbatim,hyperref}
\usepackage[all]{xy}
 
\newtheorem{theorem}{Theorem}[subsection]
\newtheorem{lemma}[theorem]{Lemma}
\newtheorem{corollary}[theorem]{Corollary}
\newtheorem{proposition}[theorem]{Proposition}

\theoremstyle{definition}
\newtheorem{definition}[theorem]{Definition}
\newtheorem{remark}[theorem]{Remark}
\newtheorem{example}[theorem]{Example}

\numberwithin{equation}{section}
\numberwithin{theorem}{section}

\newcommand{\cA}{{\cal A}}

\newcommand{\cP}{{\cal P}}

\newcommand{\cB}{{\cal B}}

\newcommand{\cI}{{\cal I}}

\newcommand{\C}{{\mathbb C}}
\newcommand{\Z}{{\mathbb Z}}
\newcommand{\N}{{\mathbb N}}

\newcommand{\B}{{\mathbb B}}

\newcommand{\R}{{\mathbb R}}
\newcommand{\Cs}{{$C^*$-al\-ge\-bra}}

\newcommand{\sh}{{$^*$-ho\-mo\-mor\-phism}}

\date{\empty}
\title{Just-infinite \Cs s and their invariants}

\author{Mikael R\o rdam\thanks{Supported  by the Danish Council for Independent Research, Natural Sciences, and the Danish National Research Foundation (DNRF) through the Centre for Symmetry and Deformation at the University of Copenhagen. Part of this work was done during my stay at the CRM attending the program ``IRP Operator Algebras: Dynamics and Interactions'' Barcelona. I thank Francesc Perera and the staff at the CRM for their hospitality.}}

\begin{document} 

\maketitle

\begin{abstract}
Just-infinite \Cs s, i.e., infinite dimensional \Cs s, whose proper quotients are finite dimensional, were investigated in \cite{GMR:JI}. One particular example of a just-infinite residually finite dimensional AF-algebras was constructed in \cite{GMR:JI}. In this paper we extend that construction by showing that each infinite dimensional metrizable Choquet simplex is affinely homeomorphic to the trace simplex of a  just-infinite residually finite dimensional \Cs. The trace simplex of any unital residually finite dimensional \Cs{} is hence realized by a just-infinite one. 
We determine the trace simplex of the particular residually finite dimensional AF-algebras constructed in \cite{GMR:JI}, and we show that it has precisely one extremal trace of type II$_1$.

We give a complete description of the Bratteli diagrams corresponding to residually finite dimensional AF-algebras. We show that a modification of any such Bratteli diagram, similar to the modification that makes an arbitrary Bratteli diagram simple, will yield a just-infinite residually finite dimensional AF-algebra. 
\end{abstract}

\section{Introduction} Just-infinite \Cs s were introduced and studied in \cite{GMR:JI} to establish an analogue of just infinite groups, which are infinite groups whose proper quotients are finite; and to examine possible connections between the two. It was shown in \cite{GMR:JI} that a separable \Cs{} is just-infinite if and only if either it is simple (and infinite dimensional), or it is an essential extension of a simple infinite dimensional \Cs{} by a finite dimensional one, or it is residually finite dimensional (RFD) and just-infinite. In the latter case its primitive ideal space is the space $Y_\infty = \{0\} \cup \N$  equipped with the (non-Hausdorff) topology making $\{0\}$ dense and all other singletons closed. In other words, the primitive ideal space of a separable RFD just-infinite \Cs{} $\cA$ is $\{0,I_1,I_2,I_3, \dots\}$, where $\cA/I_j$ is simple and finite dimensional, hence a full matrix algebra, say $M_{k_j}$. The sequence $\{k_j\}$ is called the \emph{characteristic sequence} for $\cA$, and coincides with the set (with multiplicities) of dimensions of irreducible finite dimensional representations of $\cA$. Its distinguished feature  in the just-infinite case is that it is countable and tends to infinity, as shown in \cite{GMR:JI}.

A concrete example of a RFD just-infinite AF-algebra was constructed in \cite{GMR:JI}.

The universal \Cs{} $C^*(G)$ of a discrete group $G$ is just-infinite if and only if the group algebra $\C[G]$ has a unique $C^*$-norm and is just-infinite as a $^*$-algebra. The latter implies that the group $G$ is just infinite, but the reverse implication does not hold. The group algebra $\C[G]$ has a unique $C^*$-norm if $G$ is locally finite. It was shown in \cite{BGS-2016} (see also \cite{GS-2017}) that there exist residually finite locally finite groups $G$ such that $\C[G]$ is just infinite (as a $^*$-algebra), and consequently, $C^*(G)$ is RFD and just-infinite, thus demonstrating that such \Cs s can arise from groups. 

We show in Section~\ref{sec:simplex} that each infinite dimensional metrizable Choquet simplex arises as the trace simplex of a RFD just-infinite AF-algebra. We prove this using a result of Lazar and Lindenstrauss, \cite{LazLin-1971}, that each such simplex is the inverse limit of finite dimensional simplices with surjective affine connecting mappings. The simplex of tracial states on a unital infinite dimensional RFD \Cs{} is necessarily infinite dimensional, and so it  also arises from a just-infinite RFD \Cs, in fact even an AF-algebra. In Section~\ref{Sec:example} we give concrete examples of infinite dimensional Choquet simplicies arising in this way, and show that the trace simplex of the RFD  just-infinite AF-algebra constructed in \cite[Section 4.1]{GMR:JI} is equal to $\Delta_\infty$, the Bauer simplex with extreme boundary equal to the one-point compactification of $\N$. The extreme trace corresponding to the point at infinity is of type II$_1$, and the other extremal traces are of type I.

In Section~\ref{Sec:Bratteli} we give a complete description of all RFD AF-algebras in terms of their Bratteli diagrams, and we describe which of these Bratteli diagrams correspond to RFD just-infinite AF-algebras. Our description suggests that the class of RFD just-infinite \Cs s is rather large, and that the inclusion of RFD just-infinite \Cs s inside the class of all RFD \Cs s is similar to the inclusion of simple \Cs s inside the class of all \Cs s.

I thank Anatoly Vershik for several useful conversations on topics of this paper, and I thank Stuart White for explaining a construction contained in Proposition~\ref{prop1}~(ii). 

\section{Preliminaries} \label{Sec:prelim}

\noindent We review here background material needed to prove the main results of our paper. 

The simplex of tracial states on a unital \Cs{} $\cB$ is denoted by $T(\cB)$. To each $\tau \in T(\cB)$ one has the GNS representation $\pi_\tau$ of $\cB$ on a Hilbert space $H_\tau$ and its associated von Neumann algebra $\pi_\tau(\cB)'' \subseteq B(H_\pi)$. The trace $\tau$ extends to a (faithful) tracial state on $\pi_\tau(\cB)''$, which therefore is a finite von Neumann algebra. Moreover, $\pi_\tau(\cB)''$ is a factor if and only if $\tau$ belongs to $\partial_\mathrm{e} T(\cB)$, the set of extreme points in $T(\cB)$.  In this case $\tau$ is said to be of type I$_n$, $1 \le n < \infty$, respectively, of type II$_1$, if $\pi_\tau(\cB)''$ is a factor of type I$_n$, respectively, of type II$_1$. 

As mentioned in the introduction, a separable RFD just-infinite \Cs{} $\cB$  has countably many non-zero primitive ideals $\{I_j\}_{j=0}^\infty$; and  each quotient $\cB/I_j$ is isomorphic to  a full matrix algebra, say $M_{k_j}$, for some integer $k_j \ge 1$.  Let $\pi_j \colon \cB \to \cB/I_j$ be the quotient mapping, and let $\tau_j$ be the tracial state obtained by composing $\pi_j$ with the unique normalized tracial state on $\cB/I_j \cong M_{k_j}$. Then $\pi_{\tau_j}(\cB) = \pi_{\tau_j}(\cB)'' = M_{k_j}$, which shows that $\tau_j$ is an extremal trace on $\cB$ of type I$_{k_j}$. 

\begin{proposition} \label{prop:typesoftraces}
 Let $\cB$ be a separable unital RFD just-infinite \Cs, and let $\{\tau_j\}_{j=0}^\infty$ be the extremal traces on $\cB$ defined above. If $\tau$ is an extremal trace on $\cB$, then either $\tau = \tau_j$, for some $j \ge 0$, in which case $\tau$ is of type I$_{k_j}$, or $\tau$ is of type II$_1$.
\end{proposition}

\begin{proof} Let $I$ be the kernel of the GNS-representation $\pi_{\tau}$. If $I \ne 0$, then $\cB/I$ is finite dimensional, in which case $\pi_\tau(\cB)'' = \pi_\tau(\cB)$ is isomorphic to $\cB/I$. As $\tau$ is extreme, $\cB/I$ is a factor, and therefore necessarily a full matrix algebra. Hence $I$ is a primitive ideal, so $I = I_j$, for some $j \ge 0$, which again entails that $\tau = \tau_j$, because the trace on a full matrix algebra is unique. Suppose next that $I=0$. Then $\pi_\tau(\cB)''$ is an infinite dimensional finite factor, which entails that it is a factor of type II$_1$.
\end{proof}

\noindent To prove some of our results about the trace simplex of a RFD just-infinite \Cs{} we need the following standard approximate intertwining result, stated here for compact convex sets. For completeness of the exposition  we include its proof. 
Equip the set of functions between compact metric spaces $K$ and $K'$ with the uniform metric: $d_\infty(f,g) = \sup \{d_{K'}(f(x),g(x)) : x \in K\}$.

\begin{proposition} \label{prop:intertwining}
Suppose we have a system of two inverse limits of compact convex metric spaces:
$$
\xymatrix@C-0.7pc@R+0.5pc{K_0 && \ar[ll]_-{f_0} \ar[dl]^-{\rho_0} K_1 && \ar[ll]_-{f_1} \ar[dl]^-{\rho_1} K_2 && \ar[ll]_-{f_2} \cdots & \ar[l] \ar@{-->}[d]<-.5ex>_-\rho \ K \\
& K'_0 \ar[ul]_-{\rho'_0} && \ar[ll]^-{f'_0} \ar[ul]_-{\rho'_1} K'_1 && \ar[ll]^-{f'_1} \ar[ul]_-{\rho'_2} K'_2 & \ar[l] \cdots & \ar[l] \ar@{-->}[u]<-.5ex>_-{\rho'=\rho^{-1}}  K',}
$$
where $f_j$ and  $f'_j$ are contractive affine maps, where $K$ and $K'$ are the inverse limits of the sequences in the top and bottom rows, respectively, and where $\rho_j, \rho'_j$ are contractive affine maps making the diagram above an \emph{approximate intertwining}, i.e., 
\begin{equation} \label{eq:intertwining}
\sum_{j=0}^\infty d_\infty(\rho'_j \circ \rho_{j}, f_j) < \infty, \qquad 
\sum_{j=0}^\infty d_\infty(\rho_j \circ \rho'_{j+1}, f'_j) < \infty.
\end{equation}
Then there exists an affine homeomorphism $\rho \colon K \to K'$, with inverse $\rho'$.
\end{proposition}

\begin{proof} For $0 \le i < j$, let $f_{j,i} \colon K_j \to K_i$ and $f'_{j,i} \colon K'_j \to K'_i$ be the contractive affine maps obtained by composing the maps $f_n$ and $f'_n$, respectively, and let $f_{\infty,j} \colon K \to K_j$ and  $f'_{\infty,j} \colon K' \to K'_j$ be the continuous affine maps associated with the two inverse limits. For each $0 \le i < j$, define $\sigma_{i,j} \colon K \to K'_i$ and $\sigma'_{i,j} \colon K' \to K_i$ by
$$
\sigma_{i,j} = f'_{j,i} \circ \rho_{j} \circ f_{\infty,j+1}, \qquad 
\sigma'_{i,j} = f_{j,i} \circ \rho'_{j} \circ f'_{\infty,j}.
$$
It follows from \eqref{eq:intertwining} that the sequences $\{\sigma_{i,j}\}_{j=i+1}^\infty$ and $\{\sigma'_{i,j}\}_{j=i+1}^\infty$ are Cauchy, and hence convergent, with respect to the uniform metric $d_\infty$. Their limits $\sigma_i \colon K \to K'_i$ and $\sigma'_i \colon K' \to K_i$ are continuous affine maps satisfying $f'_i \circ \sigma_{i+1} = \sigma_i$ and $f_i \circ \sigma'_{i+1} =\sigma'_i$, for all $i \ge 0$.  Therefore they factor through continuous affine maps $\rho \colon K \to K'$ and $\rho' \colon K' \to K$, that is, $f'_{\infty,i} \circ \rho = \sigma_i$ and $f_{\infty,i} \circ \rho' = \sigma'_i$, for all $i \ge 0$. 

We must still show that $\rho$ and $\rho'$ are inverses to each other; and to do so we show that $\rho' \circ \rho = \mathrm{id}_K$ (it follows in a similar way that $\rho \circ \rho' = \mathrm{id}_{K'}$). By properties of inverse limits, if $x,y \in K$ are such that $f_{\infty,i}(x) = f_{\infty,i}(y)$ for all $i$, then $x=y$. Hence it suffices to show that $(f_{\infty,i} \circ \rho \circ \rho') (x) = f_{\infty,i}(x)$ for all $x \in K$ and all $i$. Now,
\begin{eqnarray*}
(f_{\infty,i} \circ \rho' \circ \rho) (x) &=& (\sigma'_i \circ \rho)(x) \\
&=& \lim_{j \to \infty} (\sigma'_{i,j} \circ \rho)(x) \\
&=& \lim_{j \to \infty} (f_{j,i} \circ \rho'_j \circ f'_{\infty,j} \circ \rho)(x)\\
&=& \lim_{j \to \infty} (f_{j,i} \circ \rho'_j \circ \sigma_j)(x)\\
&=& \lim_{j \to \infty}\lim_{k \to \infty} (f_{j,i} \circ \rho'_j \circ \sigma_{j,k})(x)\\
&=& \lim_{j \to \infty}\lim_{k \to \infty} (f_{j,i} \circ \rho'_j \circ f'_{k,j} \circ \rho_k \circ f_{\infty,k+1})(x).
\end{eqnarray*}
(In the limits above---and below---it is understood that $k > j > i$.) By (2.1) it follows that 
$$ \lim_{j \to \infty}\lim_{k \to \infty} d_\infty(\rho'_j \circ f'_{k,j} \circ \rho_k, f_{j,k})=0.$$
Indeed, by the triangle inequality, one has
$$d_\infty(\rho'_j \circ f'_{k,j} \circ \rho_k, f_{j,k}) \le    \sum_{\ell = j}^{k-1} d_\infty(\rho'_\ell \circ \rho'_{\ell},f_\ell) + \sum_{\ell = j}^{k-2} d_\infty(\rho_\ell \circ \rho'_{\ell+1},f'_\ell).$$
It follows that
\begin{eqnarray*}
(f_{\infty,i} \circ \rho' \circ \rho) (x) &=& \lim_{j \to \infty}\lim_{k \to \infty} (f_{j,i} \circ \rho'_j \circ f'_{k,j} \circ \rho_k \circ f_{\infty,k+1})(x)\\ 
&=& (f_{j,i} \circ f_{j,k} \circ f_{\infty,k+1})(x) \\
&=& f_{\infty,i}(x),
\end{eqnarray*}
as desired.
\end{proof}

\noindent We immediately get the following corollary from the previous proposition.

\begin{corollary} \label{cor:intertwining}
Suppose we have an approximate intertwining of two inverse limits of compact convex metric spaces:
$$
\xymatrix@C-0.7pc@R+0.5pc{K_0 \ar@{=}[d] && \ar[ll]_-{f_0}  \ar@{=}[d]  K_1 && \ar[ll]_-{f_1} 
\ar@{=}[d] K_2 && \ar[ll]_-{f_2} \cdots && \ar[ll] \ar@{-->}[d]<-.5ex>_-\rho K \\
K_0  && \ar[ll]^-{f'_0}  K_1 && \ar[ll]^-{f'_1}  K_2 && \ar[ll]^-{f'_2} \cdots && \ar[ll] \ar@{-->}[u]<-.5ex>_-{\rho'=\rho^{-1}}  K',}
$$
where $f_j, f'_j$ are contractive affine maps satisfying $\sum_{j=0}^\infty d_\infty(f_j, f'_j) < \infty$.
Then there exists an affine homeomorphism $\rho \colon K \to K'$, with inverse $\rho'$.
\end{corollary}

\begin{remark} If we are given continuous (not necessarily contractive) affine maps $f_j$, $f'_j$ as in the corollary (or in the proposition) above, then one can always adjust the metrics on the spaces $K_n$ (and $K'_n$), $n \ge 0$, to make the maps $f_j$ and $f'_j$ contractive (and, in the case of Proposition~\ref{prop:intertwining} one can further make the affine maps $\rho_j$, $\rho'_j$ contractive). 
Again, given any continuous affine maps $f_j$, $f'_j$ as in the corollary above, one can also adjust the metrics on the spaces $K_n$ so that $\sum_{j=0}^\infty d_\infty(f_j, f'_j) < \infty$ holds. In general, one cannot do both! It is crucial that  \eqref{eq:intertwining} is satisfied with respect to contractive maps. 
\end{remark}

\noindent We end this section by describing the affine map between the trace simplices of finite dimensional \Cs s induced by a \sh. The proof of the lemma is straightforward and is omitted.

\begin{lemma} \label{lm:T(phi)}
Let 
$$\cA = \bigoplus_{j=0}^n M_{k_j}, \qquad \cB = \bigoplus_{i=0}^m M_{\ell_i}$$
be finite dimensional \Cs s, and let $\varphi \colon \cA \to \cB$ be a unital \sh{} with multiplicity matrix \footnote{If we identify $K_0(\cA)$ and $K_0(\cB)$ with $\Z^{n+1}$ and $\Z^{m+1}$, respectively,
then the multiplicity matrix $A$ of $\varphi$ is the $(m+1) \times(n+1)$ matrix over $\Z$ which represents the group homomorphism $K_0(\varphi) \colon \Z^{n+1} \to \Z^{m+1}$.} $A = (A(i,j))$, $i=0,1, \dots, m$, $j=0,1, \dots, n$, so that $A(i,j)$ is the multiplicity of the partial \sh{} $\varphi_{i,j} \colon M_{k_j} \to M_{\ell_i}$ (from the $j$th summand of $\cA$ to the $i$th summand of $\cB$).

Let $\{\tau_{\cA,j}\}_{j=0}^n$ and $\{\tau_{\cB,i}\}_{i=0}^m$ denote the extremal tracial states on $\cA$ and $\cB$ supported on the $j$th summand $M_{k_j}$ of $\cA$ and the $i$th summand $M_{\ell_i}$ of $\cB$, respectively, 
and let $T(\varphi) \colon T(\cB) \to T(\cA)$ be the affine homomorphism induced by $\varphi$.  Then
\begin{equation} \label{*}
T(\varphi)(\tau_{\cB,i}) = \tau_{\cB,i} \circ \varphi = \sum_{j=0}^n \frac{A(i,j)k_j}{\ell_i} \; \tau_{\cA,j},
\end{equation}
for all $i=0,1, \dots, m$.
\end{lemma}

\noindent Observe that $\sum_{j=0}^n A(i,j) k_j = \ell_i$, for all $i$, by the assumption that $\varphi$ is unital. This shows that the right-hand side of \eqref{*} indeed belongs to $T(\cA)$.

\section{The trace simplex of a just-infinite \Cs}

\label{sec:simplex}

\noindent For each $n \ge 0$, let $\Delta_n$ be the standard $n$-dimensional simplex with extreme boundary $\partial_{\mathrm{e}} \Delta_n = \{e_0^{(n)}, e_1^{(n)}, \dots, e_n^{(n)}\}$. 
It was shown by Lazar and Lindenstrauss, \cite[Corollary to Theorem 5.2]{LazLin-1971}, that each metrizable infinite dimensional Choquet simplex $\Delta$ is the inverse  limit of a  sequence 

\vspace{-.5cm}
\begin{equation} \label{eq0}
\xymatrix{\Delta_0 & \Delta_1 \ar[l]_{f_0} & \Delta_2 \ar[l]_{f_1} &  \Delta_3 \ar[l]_{f_2} & \cdots \ar[l] 
& \Delta \ar[l]
},
\end{equation}
where each $f_n$ is an affine surjective map. Since each extreme point of $\Delta_n$ lifts to an extreme point of $\Delta_{n+1}$ under any surjective affine map $\Delta_{n+1} \to \Delta_n$, we infer that 
\begin{equation} \label{eq1}
f_n(e_j^{(n+1)}) = e_j^{(n)}, \quad j=0,1, \dots, n; \qquad f_n(e_{n+1}^{(n+1)}) = \xi^{(n)},
\end{equation}
for some $\xi^{(n)} \in \Delta_n$ (possibly after relabelling the extreme points of the simplex $\Delta_n$).  The affine map $f_n$ is determined by \eqref{eq1}, so the Choquet simplex $\Delta$ is determined by the sequence $\{\xi^{(n)}\}_{n=0}^\infty$. Let $f_{\infty,n} \colon \Delta \to \Delta_n$ denote the canonical  continuous affine surjection associated with \eqref{eq0} satisfying $f_n \circ f_{\infty,n+1} = f_{\infty,n}$, for all $n \ge 0$.

\begin{lemma} \label{lm0} Let $\Delta$ be a Choquet simplex given as in \eqref{eq0} and \eqref{eq1} above (for some sequence of elements $\xi^{(n)} \in \Delta_n$, $n \ge 0$). Then, for each $n \ge 0$, there is a (unique) element $e_n \in \Delta$ satisfying $f_{\infty,m}(e_n) = e_n^{(m)}$, whenever $m \ge n$. Moreover, each $e_n$ is an extreme point of $\Delta$, and $\{e_n\}_{n=0}^\infty$  is dense in the extreme boundary of $\Delta$.
\end{lemma}

\begin{proof} The existence of $e_n \in \Delta$ follows from \eqref{eq1} and standard properties of inverse limits; and its uniqueness from the fact that if $x,y \in \Delta$ are such that $f_{\infty,m}(x) = f_{\infty,m}(y)$, for all sufficiently large $m$, then $x=y$. Since $f_{\infty,m}(e_n)$ is an extreme point in $\Delta_m$, for all $m \ge n$, $e_n$ must itself be an extreme point in $\Delta$. By Milman's partial converse to the Krein--Milman theorem, to show that $\{e_n\}_{n=0}^\infty$ is dense in the extreme boundary of $\Delta$, it suffices to show that the convex hull, $C$, of $\{e_n\}_{n=0}^\infty$ is dense in $\Delta$. However, $f_{\infty,m}(C) = \Delta_m$, for all $m \ge 0$ (since $f_{\infty,m}(C)$ is a convex sets that contains all the extreme points of $\Delta_m$), and this shows that $C$ is dense in $\Delta$.
\end{proof}

\noindent
Consider  an AF-algebra $\cA = \varinjlim (\cA_n, \varphi_n)$, whose Bratteli diagram is given as follows:
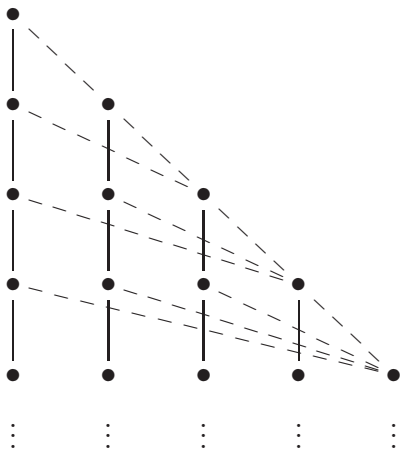
\begin{center}
\begin{equation} \label{eq:Bratteli}
\xymatrix@R-1.3pc{\bullet \ar@{-}[dd] \ar@{--}[ddr] & & & & \\ & & & &\\
\bullet \ar@{-}[dd] \ar@{--}[ddrr] & \bullet  \ar@{-}[dd] \ar@{--}[ddr] && &\\  & & & &\\
\bullet  \ar@{-}[dd] \ar@{--}[ddrrr]& \bullet  \ar@{-}[dd] \ar@{--}[ddrr]& \bullet  \ar@{-}[dd] \ar@{--}[ddr]& & \\  & & & &\\
\bullet  \ar@{-}[dd] \ar@{--}[ddrrrr] & \bullet  \ar@{-}[dd] \ar@{--}[ddrrr] & \bullet  \ar@{-}[dd] \ar@{--}[ddrr]& \bullet  \ar@{-}[dd] \ar@{--}[ddr]& \\  & & & &\\
\bullet  & \bullet   & \bullet   & \bullet  & \bullet  \\  \vdots & \vdots & \vdots & \vdots & \vdots
} 
\hspace{1.5cm}
\xymatrix@R-.53pc@C-2.3pc{\cA_0&=&M_{k_{0}} & && && &&& \\\cA_1&=&
M_{k_{0}}  & \oplus&M_{k_{1}} &&& &&& \\ \cA_2&=&
M_{k_0}  & \oplus & M_{k_1}  &\oplus & M_{k_2} & && &\\ \cA_3&=&
M_{k_0}   & \oplus & M_{k_1}   & \oplus & M_{k_2}  & \oplus  
& M_{k_3} & & \\ \cA_4&=&
M_{k_0}   & \oplus & M_{k_1}  & \oplus  & M_{k_2} & \oplus  & M_{k_3} & \oplus  & M_{k_4} \\ 
}
\end{equation}
\end{center}
where the dotted edge from the vertex at position $(n,j)$ to the vertex at position $(n+1,n+1)$ has multiplicity $m^{(n)}_j \ge 0$. The unbroken edges all have multiplicity $1$. The  connecting maps $\varphi_n \colon \cA_n \to \cA_{n+1}$ are unital, and hence determined---up to unitary equivalence---by the Bratteli diagram. Let $\varphi_{\infty,n} \colon \cA_n \to \cA$ denote the canonical inductive limit \sh{} satisfying $\varphi_{\infty,n+1} \circ \varphi_n = \varphi_{\infty,n}$, for all $n \ge 0$.
The integer $k_0 \ge 1$ can be chosen arbitrarily, and the remaining integers $k_n$, appearing in \eqref{eq:Bratteli}, are determined by the formula
\begin{equation} \label{eq:k}
k_{n+1} = \sum_{j=0}^n m^{(n)}_j \, k_j.
\end{equation}
The Bratteli diagram \eqref{eq:Bratteli}, and hence the AF-algebra $\cA$, are thus determined by the choice of the initial integer $k_0 \ge 1$ and by the multiplicity vectors
\begin{equation} \label{eq:m}
m^{(n)} = \big(m_0^{(n)}, m_1^{(n)}, m_2^{(n)}, \dots, m_n^{(n)}\big), \qquad n \ge 0,
\end{equation}
each of which, moreover, is assumed to be non-zero (to ensure that $k_n \ge 1$, for each $n \ge 0$).

\begin{lemma} \label{lm:a} The AF-algebra $\cA$, described above, is unital and RFD. For each $j \ge 0$, there is a surjective \sh{} $\pi_j \colon \cA \to M_{k_j}$, making the diagram
$$\xymatrix@C-8pt@R-8pt{\cA_n \ar[dr]_-{\pi_j^{(n)}} \ar[rr]^-{\varphi_{\infty,n}} && \cA \ar[dl]^-{\pi_j} \\ & M_{k_j} &}$$
commutative, for each $n \ge j$, where $\pi_j^{(n)}$ is the projection onto the $j$th summand of $\cA_n$.

If each multiplicity $m_j^{(n)}$ is non-zero, then $\cA$ is just-infinite, and the primitive ideal space of $\cA$ consists of $0$ and of the ideals $\mathrm{ker}(\pi_j)$, $j \ge 0$. In particular, the characteristic sequence for $\cA$ is precisely the sequence $\{k_j\}_{j=0}^\infty$ defined in (and above) \eqref{eq:k}.
\end{lemma}

\begin{proof}  By the assumptions on the connecting maps $\varphi_n$, we see that $\pi_j^{(n+1)} \circ \varphi_n = \pi_j^{(n)}$, for all $n \ge j$, so there exists a \sh{} $\pi_j \colon \cA \to M_{k_j}$, making the diagram in the lemma commutative. As $\bigoplus_{j=0}^n \pi_j^{(n)}$ obviously is injective on $\cA_n$, for each $n \ge 0$, it follows that $\bigoplus_{j=0}^\infty \pi_j$ is injective (and hence isometric) on $\bigcup_{n=0}^\infty \varphi_{\infty,n}(\cA_n)$; and therefore isometric (and hence injective) on $\cA$. This shows that $\cA$ is RFD. 

The proof that $\cA$ is just-infinite with the stipulated primitive ideal space is analogous to the proof that the \Cs{} constructed in \cite[Section 4.1]{GMR:JI} has these property. (In fact, the \Cs{} constructed therein is precisely our \Cs{} $\cA$ in the case where $m^{(n)}_j = 1$, for all $n,j$.) These facts are also stated and proved explicitly in the more general result, Theorem~\ref{thm:ji-Bratteli}, proven later in this article. The last claim about $\{k_j\}$ follows by the definition of the characteristic sequence given above. 
\end{proof}

\noindent The trace simplex $T(\cA_n)$ has extremal points $\{\tau^{(n)}_j\}_{j=0}^n$, where $\tau^{(n)}_j$ is the normalized trace on the $j$th summand, $M_{k_j}$, of $\cA_n$. We therefore have affine isomorphisms $\chi_n \colon \Delta_n \to T(\cA_n)$  such that $\chi_n(e^{(n)}_j) =\tau^{(n)}_j$, for all $j$. Under this identification we obtain surjective affine maps $f_n \colon \Delta_{n+1} \to \Delta_n$, $n \ge 0$, and an affine homeomorphism $\chi$ making the diagram
\begin{equation}  \label{eq:b}
\xymatrix@C+.5pc{\Delta_0 \ar[d]_{\chi_0} & \Delta_1 \ar[l]_{f_0} \ar[d]_{\chi_1}& \Delta_2 \ar[l]_{f_1} \ar[d]_{\chi_2} &  \Delta_3 \ar[l]_{f_2} \ar[d]_{\chi_3} & \cdots \ar[l] & \Delta \ar[l]  \ar[d]^\chi \\
T(\cA_0) & T(\cA_1) \ar[l]^-{T(\varphi_0)} & T(\cA_2) \ar[l]^-{T(\varphi_1)} & T(\cA_3) \ar[l]^-{T(\varphi_2)} & \cdots \ar[l] & T(\cA) \ar[l]
}
\end{equation}
commutative. The lemma below, which states that the sequence in the upper row of \eqref{eq:b} is of the kind described in \eqref{eq0} and \eqref{eq1}, is just a restatement of Lemma~\ref{lm:T(phi)} in the case of the particular connecting mappings under consideration. 

\begin{lemma} \label{lm:b} In the notation of \eqref{eq:Bratteli} and \eqref{eq:b}, for each $n \ge 0$, we have
$$f_n(e_j^{(n+1)}) = e_j^{(n)}, \quad j = 0,1, \dots, n; \qquad f_n(e_{n+1}^{(n+1)}) = \zeta^{(n)},$$
where
$$\zeta^{(n)} = \Big(\frac{m_0^{(n)} k_0}{k_{n+1}}, \, \frac{m_1^{(n)} k_1}{k_{n+1}}, \cdots, \, \frac{m_n^{(n)} k_n}{k_{n+1}}\Big) \in \Delta_n.$$
\end{lemma}

\noindent With $e_j$ the extreme points of $\Delta$ defined in Lemma~\ref{lm0}, we have $\chi(e_j) = \tau_j$, where $\tau_j$ is the extremal trace on $\cA$ obtained by composing the \sh{} $\pi_j \colon \cA \to M_{k_j}$ with the normalized trace on $M_{k_j}$. In particular, $\tau_j$ is of type I$_{k_j}$. It follows from Lemma~\ref{lm0} and \eqref{eq:b} that $\{\tau_j\}_{j=0}^\infty$ is dense in the set of extremal traces on $T(\cA)$, and that all other extremal traces are of type II$_1$. 

\begin{theorem} \label{thm:Simplex-ji} The following four statements are equivalent for each metrizable Choquet simplex $\Delta$:
\begin{enumerate}
\item $\Delta$ is infinite dimensional (i.e., $\Delta \ne \Delta_n$ for all $n \ge 0$).
\item There is a unital separable  infinite dimensional RFD \Cs{} whose trace simplex is affinely homeomorphic to $\Delta$.
\item There is a unital separable RFD just-infinite \Cs{} whose trace simplex is affinely homeomorphic to $\Delta$.
\item There is a unital separable RFD just-infinite AF-algebra, arising from a Bratteli diagram of the type described in \eqref{eq:Bratteli}, with $m_j^{(n)} \ge 1$, for all $n,j$, whose trace simplex is affinely homeomorphic to $\Delta$.
\end{enumerate}
\end{theorem}

\begin{proof} (i) $\Rightarrow$ (iv). If (i) holds, then by the theorem of Lazar and Lindenstrauss, mentioned in Section~\ref{Sec:prelim}, we may realize $\Delta$ as an inverse limit as  described in \eqref{eq0} and \eqref{eq1}, with respect to some sequence of elements $\xi^{(n)} = (\xi_0^{(n)}, \xi_1^{(n)}, \dots, \xi_n^{(n)}) \in \Delta_n$,  $n \ge 0$. 

Let $m_j^{(n)} \ge 1$, $0 \le j \le n$, be a system of multiplicities defining  a RFD just-infinite  AF-algebra $\cA$, as  in \eqref{eq:Bratteli}, cf.\ Lemma~\ref{lm:a}, with $k_0=1$ (or any other value of $k_0$), and with $k_{n+1}$ given as in \eqref{eq:k}, for $n \ge 0$. We show that $m_j^{(n)}$ can be chosen such that $T(\cA)$ is affinely homeomorphic to $\Delta$. Let
$$\zeta^{(n)} = \Big(\frac{m_0^{(n)} k_0}{k_{n+1}}, \, \frac{m_1^{(n)} k_1}{k_{n+1}}, \cdots, \, \frac{m_n^{(n)} k_n}{k_{n+1}}\Big) \in \Delta_n, \quad n \ge 0,$$
and let $f'_n \colon \Delta_{n+1} \to \Delta_n$ be given by $f'_n(e_j^{(n+1)}) = e_j^{(n)}$, for $0 \le j \le n$, and $f'_n(e_{n+1}^{(n+1)}) = \zeta^{(n)}$. By \eqref{eq:b} and Lemma~\ref{lm:b}, $T(\cA)$ is affinely homeomorphic to the inverse limit of the sequence 
$$\xymatrix{\Delta_0 & \Delta_1 \ar[l]_{f'_0} & \Delta_2 \ar[l]_{f'_1} &  \Delta_3 \ar[l]_{f'_2} & \cdots \ar[l]}.
$$
The simplices $\Delta$ and $T(\cA)$ are affinely homeomorphic if $\sum_{n=0}^\infty d_\infty(f_n,f'_n) < \infty$, by Corollary~\ref{cor:intertwining}, where $d_\infty$ is the uniform metric (as in Proposition~\ref{prop:intertwining}). It is easily seen that $d_\infty(f_n,f'_n) = d(\xi^{(n)}, \zeta^{(n)})$, where $d$ is the Euclidian metric on $\Delta_n \subseteq \R^{n+1}$. 

We determine $m_j^{(n)} \ge 1$ by induction after $n \ge 0$ such that $d(\xi^{(n)}, \zeta^{(n)}) \le 2^{-n}$, for all $n \ge 0$. For  any choice of $m_0^{(0)}$ we have $\xi^{(0)} = \zeta^{(0)}=1$. Let $n \ge 1$ and suppose we have chosen $m_j^{(r)}$, for all $0 \le j \le r < n$. Choose integers $\ell_0, \ell_1, \dots, \ell_{n} \ge 1$ such that
\begin{equation} \label{eq:ell}
\Big| \frac{\ell_j}{\sum_{i=0}^{n} \ell_i} - \xi^{(n)}_j\big| < \frac{1}{2^{n} \sqrt{n}}, \quad j = 0,1, \dots, n.
\end{equation}
Let $1=k_0, k_1, \dots, k_{n}$ be given as in \eqref{eq:k} (with respect to the existing $m_j^{(r)}$). Put $K= \prod_{j=0}^{n} k_j$,  and set $m_j^{(n)} = K \ell_j/k_j$, for $j=0,1, \dots, n$. Then $k_{n+1} =\sum_{j=0}^n m_j^{(n)} k_j = K \sum_{j=0}^{n} \ell_j$, so
$$\zeta_j^{(n)} = \frac{m_j^{(n)} k_j}{k_{n+1}} = \frac{\ell_j} {\sum_{i=0}^{n} \ell_i}.$$
Hence, by \eqref{eq:ell}, we obtain that $d(\xi^{(n)}, \zeta^{(n)}) < 2^{-n}$ as desired. 

The implications (iv) $\Rightarrow$ (iii) $\Rightarrow$ (ii) are trivial.

(ii) $\Rightarrow$ (i). If $\cA$ is infinite dimensional and RFD, then there is an infinite family $\{\pi_\alpha\}$ of pairwise inequivalent finite dimensional irreducible representations of $\cA$. Let $\tau_\alpha$ be the tracial state on $\cA$ obtained by composing $\pi_\alpha$ with the unique tracial state on $\pi_\alpha(\cA)$ (which is a full matrix algebra). Then each $\tau_\alpha$ is extremal (because $\pi_\alpha(\cA) = \pi_\alpha(\cA)''$ is a factor), and they are mutually distinct, because the $\pi_\alpha$'s are inequivalent, and hence have mutually distinct kernels. The infinite set of extremal traces $\{\tau_\alpha\}$ witnesses that $T(\cA)$ is infinite dimensional.
\end{proof}

\begin{remark}[Other invariants: The ordered $K_0$-group and the characteristic sequence] \label{rem:other-invariants}
Suppose that $\cA$ is a unital just-infinite AF-algebra arising from a Bratteli diagram of the type described in \eqref{eq:Bratteli}. Modifying slightly the argument from \cite[Section 4.1]{GMR:JI} one can show that its $K$-theory $(K_0(\cA),K_0(\cA)^+,[1_\cA])$, as an ordered abelian group with distinguished order unit $[1_\cA]$, 
is isomorphic to the triple $(G,G^+,u)$, where $G$ is the subgroup of $\prod_{n=0}^\infty \Z$ consisting of all $x = (x_n)_{n=0}^\infty$ for which the identity
$$x_{n+1} = \sum_{j=0}^n m_j^{(n)} \, x_j,$$
holds eventualy, and $u=(k_0,k_1,k_2, \dots) \in G^+$ determined in \eqref{eq:k}. The order on $G$ is the one inherited from the pointwise order on $\prod_{j=0}^\infty \Z$. 

The group $G$ has the additional (non-degeneracy) property that whenever $F$ is a finite subset of $\{0,1,2, \dots\}$ and $\rho_F \colon \prod_{n=0}^\infty \Z \to \prod_{n \in F} \Z$ is the canonical projection map, then $\rho_F(G) = \prod_{n \in F} \Z$. If $F = \{j\}$, then $\rho_F$ is the homomorphism of $K_0$-groups induced by the irreducible representation $\pi_j \colon \cA \to M_{k_j}$. 

Since AF-algebras are completely classified by their ordered $K_0$-group, by Elliott's theorem, all information about $\cA$ is contained in $(G,G^+,u)$. However, our picture of $G$ is not sufficiently explicit to easily reveal this information. It would be desirable to have a more concrete picture of the ordered $K_0$-group of this class of AF-algebras. Here are two other (somehow related) questions, to which we do not know the answer: Can one characterize the class of ordered\footnote{By this we mean that $G = G^+ - G^+$, when $G$ is equipped with the order inherited from $\prod_{j=0}^\infty \Z$.} non-degenerated subgroups $G$ of  $\prod_{j=0}^\infty \Z$ which are dimension groups of a RFD just-infinite AF-algebra. Is the dimension group of an arbitrary RFD just-infinite AF-algebra of this form, i.e., is it an ordered non-degenerated subgroup of $\prod_{j=0}^\infty \Z$?

Two other invariants of RFD just-infinite \Cs s, already mentioned, are  their trace simplex, discussed in the previous theorem, and their characteristic sequence $\{k_n\}_{n=0}^\infty$, which by the discussion above coincides with the order unit $u \in G$. Whereas the trace simplex can be any infinite dimensional metrizable Choquet simplex $\Delta$ by Theorem~\ref{thm:Simplex-ji}, it is less obvious what are the possible values of the characteristic sequence. It was shown in \cite[Proposition 3.18]{GMR:JI}, that $k_n \to \infty$, as $n \to \infty$. 

For each fixed Choquet simplex $\Delta$ as above, the set of possible characteristic sequences $\{k_n\}_{n=0}^\infty$ of a RFD just-infinite \Cs{} $\cA$ with $T(\cA) = \Delta$, is uncountable. Indeed, inspecting the proof of (i) $\Rightarrow$ (iv) in Theorem~\ref{thm:Simplex-ji}, we see that, for each $n \ge 0$, there are (countably) infinitely many $(n+1)$-tuples $(\ell_0, \ell_1, \dots, \ell_n)$ satisfying \eqref{eq:ell}. These, in turn, give rise to infinitely many choices for the multiplicity vector $(m_0^{(n)}, m_1^{(n)}, \dots, m_n^{(n)})$, and hence infinitely many values of $k_{n+1}$. This results in uncountably many possibilities for the characteristic sequence $\{k_n\}$  of a RFD just-infinite AF-algebra with trace simplex $\Delta$. 
\end{remark}

\section{An example}
\label{Sec:example}

\noindent In this section we make explicit computations of the Choquet simplex $\Delta$ described in \eqref{eq0} and \eqref{eq1}, and in particular of the trace simplex of the RFD just-infinite \Cs{} constructed in \cite[Section 4.1]{GMR:JI} (and in \eqref{eq:Bratteli} with $m_j^{(n)} = 1$, for all $j,n$). 

Recall that $\Delta$ is determined by the sequence $\xi^{(n)} \in \Delta_n$, $n \ge 0$, cf.\ \eqref{eq1}.
The first class of examples we shall consider will be referred to as the \emph{stationary case}, by which we mean that the sequence of points $\xi^{(n)} \in \Delta_n$ satisfies $f_n(\xi^{(n+1)}) = \xi^{(n)}$ for all sufficiently large $n$.

\begin{lemma} \label{stationary} Let $t = \{t_n\}_{n=0}^\infty$ be a non-zero sequence of non-negative numbers. Let $n_0 \ge 0$ be the smallest integer such that $t_{n_0} \ne 0$. For each $n \ge n_0$, define
\begin{equation} \label{stat}
\xi^{(n)} = \Big(\sum_{j=0}^n t_j\Big)^{-1} (t_0,t_1, \dots, t_n) \in \Delta_n,
\end{equation}
and let $\xi^{(n)} \in \Delta_n$ be arbitrary, for $0 \le n < n_0$ (if $n_0 >0$). Then $\{\xi^{(n)}\}_{n=0}^\infty$ gives rise to a stationary sequence.
Conversely, each stationary sequence $\{\xi^{(n)}\}_{n=0}^\infty$ arises in this way.
\end{lemma}

\begin{proof} Let $t = \{t_n\}_{n=0}^\infty$ be given as in the lemma, and let $\{\xi^{(n)}\}_{n=0}^\infty$ be given as in \eqref{stat}. We show that $f_n(\xi^{(n+1)}) = \xi^{(n)}$, for all $n \ge n_0$. Let $n \ge n_0$ and put $\alpha = \sum_{j=0}^n t_j$ and $\beta = \sum_{j=0}^{n+1} t_j$. Then
\begin{eqnarray*}
f_{n}(\xi^{(n+1)}) &=& \sum_{j=0}^n \beta^{-1}t_j f_n(e^{(n+1)}_j) + \beta^{-1} t_{n+1}f_n(e^{(n+1)}_{n+1}) \\ &=& \sum_{j=0}^n \beta^{-1}t_j e^{(n)}_j + \beta^{-1}t_{n+1} \xi^{(n)} 
\; = \; \big(\beta^{-1} \alpha  + \beta^{-1}t_{n+1}\big) \xi^{(n)} = \xi^{(n)}.
\end{eqnarray*}

To prove the converse claim, it suffices to show that if $f_n(\xi^{(n+1)}) = \xi^{(n)}$, then $\xi^{(n)}$ and $\big(\xi^{(n+1)}_0,\xi^{(n+1)}_1, \dots, \xi^{(n+1)}_{n}\big)$ are proportional. This follows from the identity:
\begin{eqnarray*}
\xi^{(n)} &=& f_n(\xi^{(n+1)}) = \sum_{j=0}^n \xi_j^{(n+1)} \cdot f_n(e_j^{(n+1)}) + \xi^{(n+1)}_{n+1} \cdot f_n(e^{(n+1)}_{n+1}) \\
&=& \sum_{j=0}^n \xi^{(n+1)}_j \cdot e_j^{(n)} + \xi^{(n+1)}_{n+1} \cdot \xi^{(n)} = \big(\xi^{(n+1)}_0,\xi^{(n+1)}_1, \dots, \xi^{(n+1)}_n\big) + \xi^{(n+1)}_{n+1} \cdot \xi^{(n)}. 
\end{eqnarray*}

\vspace{-.6cm}
\end{proof}

\noindent Two non-zero sequences  $t= \{t_n\}_{n=0}^\infty$ and $t'= \{t'_n\}_{n=0}^\infty$ of non-negative numbers give rise to the same stationary sequence $\{\xi^{(n)}\}_{n=0}^\infty$, as in Lemma~\ref{stationary}, if and only if they are proportional. Hence, if $\sum_{n=0}^\infty t_n < \infty$, we may without loss of generality assume that $\sum_{n=0}^\infty t_n =1$.

Let $\Delta_\infty$ denote the infinite dimensional Bauer simplex with extreme boundary $\partial_\mathrm{e}\Delta_\infty = \{e^{(\infty)}_j : 0 \le j \le \infty\}$ equipped with the topology making it homeomorphic to $\N_0 \cup \{\infty\}$, the one-point compactification of the discrete space $\N_0$. In particular, $e^{(\infty)}_j \to e^{(\infty)}_\infty$ as $j \to \infty$. Each point  $x$ in $\Delta_\infty$ is a unique infinite convex combination $x = \sum_j \lambda_j e^{(\infty)}_j$, where $\lambda_j \ge 0$ and $\sum_j \lambda_j = 1$.

\begin{proposition} \label{prop1}  Let $\Delta$ be the Choquet simplex defined in \eqref{eq0} and \eqref{eq1} with respect to a stationary sequence $\xi^{(n)} \in \Delta_n$, $n \ge 0$.  
Let $\{t_n\}_{n=0}^\infty$ be a sequence of non-negative numbers generating $\{\xi^{(n)}\}$ as in \eqref{stat}, and let $n_0 \ge 0$ be the smallest integer such that $t_{n_0} \ne 0$.

Let  $e_\infty \in \Delta$ be determined by $f_{\infty,n}(e_\infty) = \xi^{(n)}$, for all $n \ge n_0$, and let $\{e_n : 0 \le n < \infty\}$ be the dense subset of the extreme boundary of $\Delta$ defined in Lemma~\ref{lm0}. 
\begin{enumerate}
\item If $\sum_{n=0}^\infty t_j = \infty$, then $\Delta$ is affinely homeomorphic to $\Delta_\infty$, and the extreme boundary of $\Delta$ is equal to $\{e_n : 0 \le n \le \infty\}$. Moreover, $e_n \to e_\infty$ as $n \to \infty$. 
\item If $\sum_{n=0}^\infty t_j = 1$, then the extreme boundary of $\Delta$ is equal to the set 
$\{e_n : 0 \le n < \infty\}$, and 
$$\lim_{n\to\infty} e_n = e_\infty = \sum_{j=0}^\infty t_j e_j.$$
In particular, $\Delta$ is not a Bauer simplex. 
\end{enumerate}
\end{proposition}

\begin{proof} We show first that there is a continuous affine surjective map $g \colon \Delta_\infty \to \Delta$ satisfying $g(e_j^{(\infty)}) = e_j$, for all $0 \le j \le \infty$. Thus, in particular,  $e_n \to e_\infty$ as $n \to \infty$. Consider, for each $n \ge n_0$,  the continuous affine surjective map $g_n \colon \Delta_\infty \to \Delta_n$ given by
\begin{equation} \label{eq3}
g_n(e_j^{(\infty)}) = \begin{cases}e_j^{(n)}, &  0 \le j \le n, \\ \xi^{(n)}, & n < j \le \infty. \end{cases}
\end{equation}
As $(f_n \circ g_{n+1})(e_j^{(\infty)}) =  g_{n}(e_j^{(\infty)})$, for all $0 \le j \le \infty$ and all $n \ge n_0$, we conclude that $f_n \circ g_{n+1} = g_{n}$. We therefore obtain  $g \colon \Delta_\infty \to \Delta$ satisfying $f_{\infty,n} \circ g = g_n$, for all $n \ge n_0$. It follows from Lemma~\ref{lm0} and \eqref{eq3}, that $g(e_j^{(\infty)}) = e_j$, when $0 \le j < \infty$;  and from the identity
$$f_{\infty,n}(e_\infty) =  \xi^{(n)} =  g_n(e_\infty^{(\infty)})= f_{\infty,n}(g(e_\infty^{(\infty)})),$$ 
which holds for all $n \ge n_0$,  we obtain $g(e_\infty^{(\infty)}) = e_\infty$.  

(i). Suppose that $\sum_{n=0}^\infty t_j = \infty$. We show, in this case, that $g \colon \Delta_\infty \to \Delta$ is injective and hence an affine homeomorphism.
Let 
$$y = \sum_{0 \le j \le \infty} y_j e_j^{(\infty)} \in \Delta_\infty, \qquad 
z = \sum_{0 \le j \le \infty} z_j e_j^{(\infty)} \in \Delta_\infty,$$ be given and suppose that $g(y) = g(z)$. Then $g_n(y) = g_n(z)$, for all integers $n \ge n_0$, so
$$(y_0,y_1, \dots, y_n) + (1-\sum_{j=0}^n y_j)\xi^{(n)} = (z_0,z_1, \dots, z_n) + (1-\sum_{j=0}^n z_j)\xi^{(n)},$$
for all $n \ge 0$. 
The assumption on the sequence $\{t_j\}$ implies that the $i$th entry of $\xi^{(n)}$ converges to $0$ as $n \to \infty$, for each $i \ge 0$. Fix the $i$th coordinate in the equation above and let $n \to \infty$. Then we see that $y_i = z_i$, for all $0 \le i < \infty$. As $\sum_j y_j = \sum_j z_j = 1$, this entails that $y_\infty=z_\infty$, so $y = z$, as desired.

(ii). Being surjective, $g$ maps $\partial_{\mathrm{e}} \Delta_\infty$ onto  $\partial_{\mathrm{e}} \Delta$, so $\partial_{\mathrm{e}}\Delta$ is contained in $g(\partial_{\mathrm{e}}\Delta_\infty) = \{e_j : 0 \le j \le \infty\}$. By  Lemma~\ref{lm0} we know that $e_j \in \partial_{\mathrm{e}}\Delta$ for all $0 \le j < \infty$. Now,
$$
f_{\infty,n}\big(\sum_{j=0}^\infty t_j e_j\big) =  \sum_{j=0}^\infty t_j f_{\infty,n}(e_j) 
=  \sum_{j=0}^n t_je_j^{(n)} + (1-\sum_{j=0}^n t_j)\xi^{(n)} =  \xi^{(n)}  = f_{\infty,n}(e_\infty),
$$
for all $n \ge n_0$, which implies that $e_\infty = \sum_{j=0}^\infty t_j e_j$. In particular, $g$ is not injective and $e_\infty$ is not an extreme point of $\Delta$. We already observed that $e_j \to e_\infty$, which proves the last claim in (ii).
\end{proof}

\begin{example} \label{ex1}
Let $\cA$ be the just-infinite RFD AF-algebra constructed in 
\cite[Section 4.1]{GMR:JI}, and also in \eqref{eq:Bratteli} with multiplicities $m_j^{(n)} = 1$, for all $0 \le j \le n$. It has characteristic sequence $k_0=1$, and $k_j = 2^{j-1}$, for $j \ge 1$, cf.\ \eqref{eq:k}. 

For $0 \le j < \infty$, let $\tau_j \in \partial_e(T(\cA))$, $\pi_j \colon \cA \to M_{k_j}$,  and $\chi \colon \Delta \to T(\cA)$ be as defined in Lemma~\ref{lm:a},  \eqref{eq:b}, and below Lemma~\ref{lm:b}; and recall that $\tau_j = \mathrm{Tr} \circ \pi_j$, that $\chi(e_j) = \tau_j$, and that $\{\tau_j\}_{j=0}^\infty$ is dense in $\partial_{\mathrm{e}} T(\cA)$. 

In the notation of Lemma~\ref{lm:b}, we have 
$$
\zeta^{(n)} = \big( \frac{k_0}{k_{n+1}}, \frac{k_1}{k_{n+1}}, \dots, \frac{k_n}{k_{n+1}}\big) 
= \big(\sum_{j=0}^{n} t_j\big)^{-1} \big(t_0,t_1,t_2, \dots, t_n\big),
$$
when $t_j = k_j$, for all $j \ge 0$. In other words, we are in the stationary case covered in Proposition~\ref{prop1}~(i). Hence $\Delta=\Delta_\infty$ and $\partial_{\mathrm{e}} \Delta = \{e_j : 0 \le j \le \infty\}$ (as in Proposition~\ref{prop1}). It follows that $\partial_{\mathrm{e}} T(\cA) = \{ \tau_j : 0 \le j \le \infty\}$, where $\tau_\infty = \chi(e_\infty)$ is  an extremal trace of type II$_1$, by Proposition~\ref{prop:typesoftraces}, and $\tau_j \to \tau_\infty$ as $j \to \infty$.

The extremal type II$_1$ trace $\tau_\infty$ can be described in a  little more detail as follows using Proposition~\ref{prop:typesoftraces}. Since $f_{\infty,n}(e_\infty) = \zeta^{(n)}$, for all $n \ge n_0$ ($n_0 = 0$ in our case), the restriction $\tau_\infty^{(n)}$ of $\tau_\infty$ to the subalgebra $\cA_n$ is given as
$$\tau_\infty^{(n)} =2^{-n} \big(\tau_0^{(n)} + \tau_1^{(n)} + 2 \tau_2^{(n)} + 4 \tau_3^{(n)} + \cdots + 2^{n-1} \tau_n^{(n)}\big),$$
for all $n \ge 1$, where $\tau_j^{(n)}$ is the $j$th extremal trace (cf.\ the comments above \eqref{eq:b}) on $\cA_n$. 

In conclusion, we have shown that the trace simplex of $\cA$ is the Bauer simplex $\Delta_\infty$, and each extremal trace of $\cA$ has been identified. In particular, $\cA$ has precisely one extremal trace, $\tau_\infty$, of type II$_1$, which is determined by the equation above. 
\end{example}

\begin{example} \label{ex2}
Let $\Delta$ be the (non-Bauer) Choquet simplex arising as in Proposition~\ref{prop1}~(ii) with respect to a sequence $\{t_n\}_{n=0}^\infty$ of positive numbers adding up to $1$. Then $\partial_{\mathrm{e}} \Delta = \{e_j : 0 \le j < \infty\}$, where $e_j$ is as defined in Lemma~\ref{lm0}, and $e_n \to \sum_{j=0}^\infty t_j e_j$ as $n \to \infty$. 

By Theorem~\ref{thm:Simplex-ji} we can realize $\Delta$ as the trace simplex, $T(\cA)$, of a RFD just-infinite AF-algebra $\cA$ with Bratteli diagram of the type described in \eqref{eq:Bratteli} (for suitable multiplicities $m^{(n)}_j \ge 1$, $0 \le j \le n$). In the notation of \eqref{eq:b}, and 
as in Example~\ref{ex1} above, $\tau_j$ is an extremal trace on $\cA$ of type I$_{k_j}$, and $\chi(e_j) = \tau_j$, for $0 \le j < \infty$. The set of extremal traces on $\cA$ is therefore equal to $\{\tau_0, \tau_1, \tau_2, \dots \}$, and
$$\lim_{n \to \infty} \tau_n = \sum_{j=0}^\infty t_j \tau_j,$$
in the weak$^*$ topoology. Hence $\cA$ has no extremal trace of type II$_1$. In particular, the bidual $\cA^{**}$ has no central portion of type II$_1$ (and $\cA$ has no representation of type II$_1$). 

The Bratteli diagram for $\cA$ is determined by the multiplicities $m^{(n)}_j$, which again can be derived (although not uniquely) from the given sequence $\{t_n\}_{n=0}^\infty$, as in the proof of Theorem~\ref{thm:Simplex-ji}. This construction simplifies when  all $t_n$ are rational numbers, in which case the proof of Theorem~\ref{thm:Simplex-ji} yields a  recipe for finding the multiplicities $m^{(n)}_j$ such that $\zeta^{(n)} = \xi^{(n)}$, for all $n \ge 0$. (We can choose the integers $\ell_j$ such that the quantity on the left-hand side of \eqref{eq:ell} is identically zero.) The affine homeomorphism $\chi \colon \Delta \to T(\cA)$ can then be obtained without using the approximate intertwining of Corollary~\ref{cor:intertwining}.
\end{example}

\begin{example} \label{ex:Bauer}
It follows from the theorem of Lazar and Lindstrauss, mentioned at the beginning of Section~\ref{sec:simplex}, that any infinite dimensional metrizable Choquet simplex is an inverse limit as in \eqref{eq0} and \eqref{eq1}. In particular, each Bauer simplex $\cP(X)$ of probability measures on an infinite metrizable compact Hausdorff space $X$ arises in this way. 

We indicate here a direct way to see this. For each integer $n \ge 0$, choose $x_n \in X$, an open cover $\{U_j^{(n)}\}_{j=0}^n$  of $X$, and a partition $\{\varphi_j^{(n)}\}_{j=0}^n \subseteq C(X)$ of  the unit subordinate to $\{U_j^{(n)}\}_{j=0}^n$ such that:
\begin{enumerate}
\item $\varphi_j^{(n)}(x_j) = 1$, for each $n \ge 0$ and for each $0 \le j \le n$,
\item $\mathrm{span}\{\varphi_j^{(n)} : n \ge 0, \; 0 \le j \le n\}$ is dense in $C(X)$. 
\end{enumerate}
Condition (i) implies that $x_j \in U_j^{(n)}$, for all $n \ge 0$, and that $\varphi_j^{(n)}(x_i) = \delta_{i,j}$, for all $n \ge 0$ and for all $i,j=0,1, \dots, n$.
Put $X_n = \{x_0,x_1, \dots, x_n\}$. Define ucp maps $G_n \colon C(X_n) \to C(X)$, $n \ge 0$, by
$$G_n(h) = \sum_{j=0}^n h(x_j) \varphi_j^{(n)}, \qquad h \in C(X_n),$$
and define ucp maps  $F_n \colon C(X_{n}) \to C(X_{n+1})$ by $F_n(h) = G_n(h)|_{X_{n+1}}$, for $n \ge 0$.  Let 
$f_n \colon \cP(X_{n+1}) \to \cP(X_n)$ and $g_n \colon \cP(X) \to \cP(X_n)$
be the continuous affine maps induced by $F_n$ and $G_n$, respectively, and let $\Delta$ be the inverse limit of the sequence
$$\xymatrix{\cP(X_0) & \cP(X_1) \ar[l]_{f_0} & \cP(X_2) \ar[l]_{f_1} &  \cP(X_3) \ar[l]_{f_2} & \cdots \ar[l]},$$
with associated affine continuous maps $f_{\infty,n} \colon \Delta \to \cP(X_n)$. The extreme boundary of $\cP(X_n)$ is equal to $\{\delta_{x_j}\}_{j=0}^n$, where 
$\delta_x$ denotes the Dirac measure in $x \in X$. One can now verify that $f_n(\delta_{x_j}) = \delta_{x_j}$, for $j = 0,1, \dots, n$, and that 
$$\xi^{(n)} := f_n(\delta_{x_{n+1}})  = \sum_{j=0}^n \varphi_j^{(n)}(x_{n+1}) \, \delta_{x_j}.$$
In other words, $\Delta$ arises as in \eqref{eq0} and \eqref{eq1} with $\xi^{(n)}$ given as above.
Since each $f_n$ is surjective, and since $f_{n+1} \circ g_{n+1} = g_n$, for all $n \ge 0$, there is a continuous affine surjective map $g \colon \cP(X) \to \Delta$ such that $f_{\infty,n} \circ g = g_n$, for all $n \ge 0$. 

To see that $g$ is injective, if $\mu, \nu \in \cP(X)$ are such that $g(\mu) = g(\nu)$, then $g_n(\mu) = g_n(\nu)$, for all $n \ge 0$, which implies that $\mu(\varphi_j^{(n)}) = \nu(\varphi_j^{(n)})$, for all $0 \le j \le n$. Hence $\mu=\nu$ by (ii).
\end{example}

\section{The Bratteli diagrams of general RFD AF-algebras and of RFD just-infinite AF-algebras}
\label{Sec:Bratteli}

In this section we shall give a complete description of which Bratteli diagram give rise to RFD AF-algebras and among those, which give rise to RFD just-infinite AF-algebras, see Theorems \ref{thm:Bratteli-rfd} and \ref{thm:ji-Bratteli} below.  Our results show that passing from a Bratteli diagram of an arbitrary RFD AF-algebra to one of a just-infinite RFD AF-algebra, is much like passing from a general Bratteli diagram (of an arbitrary AF-algebra) to a simple one. 

It is notationally more convenient to formulate these results in terms of direct limits of simplical groups, which carry the same information as Bratteli diagrams. Accordingly, we consider a direct limit of the form
\begin{equation} \label{eq:dl}
\xymatrix{(\Z^{m_1}, u_1) \ar[r]^-{A_1} & (\Z^{m_2}, u_2) \ar[r]^-{A_2} & (\Z^{m_3}, u_3) \ar[r]^-{A_3} & \cdots,}
\end{equation}
where $m_n \ge 1$ are integers, where each $A_n$ is an $m_{n+1} \times m_{n}$ matrix with non-negative integer coefficients, which is non-degenerate in the sense that all its rows and columns  are non-zero, and where $u_n$ is an order unit for $\Z^{m_n}$, satisfying $A_n u_n \le u_{n+1}$. (The order of $\Z^m$ is the usual one given by $x \le y$ if $x_j \le y_j$, for all $1 \le j \le m$.) 

The AF-algebra associated with the sequence \eqref{eq:dl} is the inductive limit of the sequence
\begin{equation} \label{eq:algebras}
\xymatrix{\cA_1 \ar[r]^{\varphi_1} & \cA_2 \ar[r]^{\varphi_2} & \cA_3 \ar[r]^{\varphi_3} & \cdots,}
\end{equation}
of finite dimensional \Cs s $\cA_n=  \bigoplus_{i=1}^{m_n} \cA_{n,i}$, where each $\cA_{n,i}$ is (isomorphic to) the full matrix algebra $M_{u_n(i)}$. The connecting mapping $\varphi_n \colon \cA_n \to \cA_{n+1}$ is determined, up to unitary equivalence, by the property the partial map $\cA_{n,k} \to \cA_{n+1,\ell}$ has multiplicity $A_n(\ell,k)$. If $A_n u_n = u_{n+1}$, then $\varphi_n$ is unital. 

\begin{definition} \label{def:RFD-JI}
A sequence of simplicial groups, as in \eqref{eq:dl}, will be said to have \emph{property (RFD)} if there exists a strictly increasing sequence $\{r_n\}_{n=1}^\infty$ of integers satisfying $1 \le r_n \le m_n$, such that the matrices $A_n$, $n \ge 1$, take the special form
\begin{equation} \label{eq:E_n}
A_n = \begin{pmatrix}I_{r_n} & 0\\A_n^{(2,1)} &A_n^{(2,2)} \\A_n^{(3,1)} &A_n^{(3,2)}  \end{pmatrix},\end{equation}
with respect to the block-decomposition 
$$m_{n+1} = r_n + (r_{n+1}-r_n) + (m_{n+1}-r_{n+1}), \qquad m_n = r_n + (m_n-r_n),$$
and, moreover, each column of the $(r_{n+1}-r_n) \times (m_n-r_n)$ matrix $A_n^{(2,2)}$ is non-zero. 

In addition, we require that $u_{n+1}(j) = u_n(j)$,  whenever $0 \le j \le r_n$.
Accordingly, there is a sequence $\{k_j\}_{j=1}^\infty$ of positive integers such that $u_n(j) = k_j$, whenever $1 \le j \le r_n$. 

If, moreover, each entry of each of the block matrices $A_n^{(i,j)}$, $n \ge 1$, $i=2,3$, $j=1,2$, from \eqref{eq:E_n} is non-zero, then we say that the sequence \eqref{eq:dl} has \emph{property (RFD-JI)}.
\end{definition}

\noindent If $r_n = m_n$, then the second column in \eqref{eq:E_n} disappears, and so does the condition on $A_n^{(2,2)}$. If $r_{n+1} = m_{n+1}$, then the third row in \eqref{eq:E_n} will disappear. The importance of allowing for the possibility that each $r_n$ is strictly smaller than $m_n$ (in the context of Theorem~\ref{thm:Bratteli-rfd}) is discussed in Example~\ref{ex:CAR-quotient}.

\begin{lemma} \label{lm:RFD-maps} Let $\cA = \varinjlim(\cA_n,\varphi_n)$ be the inductive limit of a sequence of finite dimensional \Cs s $\cA_n$ corresponding to a sequence of simplicial groups \eqref{eq:dl} with property (RFD). Then, in the notation of Definition~\ref{def:RFD-JI}, $\cA_n = \bigoplus_{j=1}^{m_n} \cA_{n,j}$, where $\cA_{n,j} = M_{k_j}$, when $1 \le j \le r_n$, (and where $\cA_{n,j} = M_{u_n(j)}$, for $r_n < j \le m_n$). 

Moreover, for each $j \ge 1$, there a surjective \sh{} $\pi_j \colon \cA \to M_{k_j}$ making the diagram
$$\xymatrix{\cA_n \ar[r]^-{\pi_j^{(n)}} \ar[d]_-{\varphi_{\infty,n}} & \cA_{n,j} \ar@{=}[d] \\
\cA \ar[r]_{\pi_j} & M_{k_j}}$$
commutative, for all $1 \le j \le r_n$, where $\pi_j^{(n)}$ is the projection onto the $j$th summand of $\cA_n$.
\end{lemma}

\begin{proof} The first part of the lemma follows from directly from Definition~\ref{def:RFD-JI}. Since $A_n(j,j) = 1$, when $1 \le j \le r_n$, by \eqref{eq:E_n}, the map $\cA_{n,j} \to \cA_{n+1,j}$ induced by $\varphi_n$ is an isomorphism, so upon choosing a suitable  identification between  $\cA_{n,j}$ and $M_{k_j}$, we may  assume that $\pi_j^{(n+1)} \circ \varphi_n = \pi_j^{(n)}$, whenever $1 \le j \le r_n$. This provides us with the existence of $\pi_j$.
\end{proof}

\begin{lemma} \label{lm:separable} Let $\cA$ be a separable infinite dimensional RFD \Cs. Then there exists a separating sequence $\{\nu_j\}_{j=1}^\infty$ of pairwise inequivalent irreducible finite dimensional representations of $\cA$.
\end{lemma}

\begin{proof} Let $\{a_j\}_{j=1}^\infty$ be a dense subset of the set  elements in $\cA$ of norm one. Let $\mathcal{P}$ be the set of all irreducible finite dimensional representations of $\cA$. By the assumption that $\cA$ is RFD, the direct sum of all the representations from $\mathcal{P}$ is faithful and therefore isometric. Hence $\|a\| = \sup\{\|\nu(a)\| : \nu \in \mathcal{P}\}$ for all $a \in \cA$. In particular, for each $j \ge 1$, we can find $\nu_j \in \mathcal{P}$ such that $\|\nu_j(a_j)\| \ge 1/2$. If $a \in \cA$ is an arbitrary element of norm one, then $\|a-a_j\| < 1/2$ for some $j \ge 1$, whence $\|\nu_j(a)\| > \|\nu_j(a_j)\| - 1/2 \ge 0$, so $\nu_j(a) \ne 0$. This proves that $\{\nu_j\}_{j=1}^\infty$  is separating. Finally, by passing to a subset of the sequence $\{\nu_j\}_{j=1}^\infty$,  we can arrange that the representations $\nu_j$ are pairwise inequivalent. 
\end{proof}

\begin{theorem} \label{thm:Bratteli-rfd}
Any AF-algebra associated with a sequence of simplicial groups \eqref{eq:dl} with property (RFD) is itself RFD. Conversely, any RFD infinite dimensional separable AF-algebra is realized by a sequence of simplical groups which is (RFD).
\end{theorem}

\begin{proof} Assume first that $\cA$ is an AF-algebra, which is the direct limit of the sequence \eqref{eq:algebras} arising from a RFD sequence of simplicial groups \eqref{eq:dl}. We must show that $\cA$ is RFD. For this it suffices to show that the sequence $\{\pi_j\}_{j=1}^\infty$ of irreducible finite dimensional representations of $\cA$, found in Lemma~\ref{lm:RFD-maps} above, is separating. This will follow if we can show that the sequence $\{\pi_j \circ \varphi_{\infty,n}\}_{j=1}^\infty$ is separating for $\cA_n$, for each $n \ge 1$. By the assumptions on the multiplicity matrix $A_n$, that each column in $A_n^{(2,2)}$ is non-zero and that $A_n(j,j) = 1$, for $1 \le j \le r_n$, we obtain that the \sh{}
$$\bigoplus_{j=1}^{r_{n+1}} \pi_j^{(n+1)} \circ \varphi_n \colon \cA_n \to \bigoplus_{j=1}^{r_{n+1}} \cA_{n+1,j},$$
is injective. As  $\pi_j \circ \varphi_{\infty,n} = \pi_j^{(n+1)} \circ \varphi_n$, $1 \le j \le r_{n+1}$, it follows that $\{\pi_j \circ \varphi_{\infty,n}\}_{j=1}^{r_{n+1}}$ is separating for $\cA_n$, as desired.

Suppose now that $\cA$ is a separable infinite dimensional RFD AF-algebra. Choose an increasing sequence $\cA_1 \subset \cA_2 \subset \cA_3 \subset \cdots$ of finite dimensional sub-\Cs s of $\cA$ with dense union, and choose a separating countably infinite family $\{\nu_j\}_{j=1}^\infty$ of irreducible pairwise inequivalent finite dimensional representations of $\cA$, cf.\ Lemma~\ref{lm:separable}. Let $k_j$ denote the dimension of the representation $\nu_j$, so that $\nu_j(\cA) = M_{k_j}$. 

We claim that there are increasing sequences of integers $1 \le n_1 < n_2 < n_3 < \cdots$ and $1 = r_1 < r_2 < r_3 < \cdots$ such that:
\begin{enumerate}
\item $\big(\bigoplus_{j=1}^{r_k} \nu_j\big)(\cA_{n_k}) = \big(\bigoplus_{j=1}^{r_k}  \nu_j\big)(\cA)$, for each $k \ge 1$,
\item the restriction of $\bigoplus_{j=1}^{r_k} \nu_j$ to $\cA_{n_{k-1}}$ is faithful, for each $k \ge 2$.
\end{enumerate}
To see this note that for each $r \ge 1$ there exists $n \ge 1$ such that $\big(\bigoplus_{j=1}^{r} \nu_j\big)(\cA_n) = \big(\bigoplus_{j=1}^{r}\nu_j\big)(\cA)$; and for each $n \ge 1$ there exists $r \ge 1$ such that the restriction of $\bigoplus_{j=1}^{r} \nu_j$ to $\cA_{n}$ is faithful.  We can therefore find $n_1 \ge 1$ such that (i) holds with $r_1=1$. Next, we find $r_2 > r_1$ such that (ii) holds for $k=2$. Proceed like this to construct the desired sequences. 

Upon passing to a subsequence of $\cA_1 \subset \cA_2 \subset \cdots$, we may assume that $n_k = k$,  for all $k \ge 1$, so that $\big(\bigoplus_{j=1}^{r_n} \nu_j\big)(\cA_{n}) = \big(\bigoplus_{j=1}^{r_n}  \nu_j\big)(\cA)$, for all $n \ge 1$, and the restriction of $\bigoplus_{j=1}^{r_{n+1}} \nu_j$ to $\cA_{n}$ is faithful, for all $n \ge 1$.

Since the $\nu_j$'s are pairwise inequivalent, we obtain that
$$\big(\bigoplus_{j=1}^{r_n} \nu_j\big)(\cA_{n}) = \big(\bigoplus_{j=1}^{r_n}  \nu_j\big)(\cA) \cong \bigoplus_{j=1}^{r_n}  \nu_j(\cA) = \bigoplus_{j=1}^{r_n} M_{k_j},$$
for each $n \ge 1$. Being a finite dimensional \Cs, $\cA_n$ is the direct sum of the image and of the kernel of the \sh{} $\bigoplus_{j=1}^{r_n} \nu_j$. The kernel, if non-zero, is equal to a direct sum 
 $\bigoplus_{j=r_n+1}^{m_n} \cA_{n,j}$ of full matrix algebras $\cA_{n,j}$, for some $m_n > r_n$. If we set $\cA_{n,j} = \nu_j(\cA_n) = M_{k_j}$, for $1 \le j \le r_n$, then we obtain that $\cA_n$ is (isomorphic to) $\bigoplus_{j=1}^{m_n} \cA_{n,j}$ (with $m_n = r_n$ if the kernel of $\bigoplus_{j=1}^{r_n} \nu_j$ is zero). 

It remains to verify that the inclusion mapping $\cA_n \subset \cA_{n+1}$ has multiplicity matrix $A_n$ which satisfies the (RFD) property given in Definition~\ref{def:RFD-JI}. The restriction of $\nu_i$ to the direct summand $\cA_{n,j}$ is an isomorphism if $1 \le i=j \le r_n$, and zero if $1 \le i \le r_n$ and $j \ne i$. The multiplicity, $A_n(i,j)$, of the \sh{}
$$\xymatrix{\cA_{n,j} \ar@{^{(}->}[r] & \cA_n \ar@{^{(}->}[r] & \cA_{n+1} \ar[r]^-{\nu_i} & \cA_{n+1,i}}$$
is therefore one if $1 \le i=j \le r_n$, and zero if $1 \le i \le r_n$ and $j \ne i$. This shows that $A_n$ is of the form given in \eqref{eq:E_n}. We still need to show that each column of the block matrix $A_n^{(2,2)}$ is non-zero.  By (ii) we know that the map $\rho = \bigoplus_{j=1}^{r_{n+1}} \nu_j$ is injective on $\cA_n$, i.e.,
$$\xymatrix{\cA_n \ar@{^{(}->}[r]& \cA_{n+1} \ar[r]^-{\rho} & \bigoplus_{i=1}^{r_{n+1}} \cA_{n+1,i}}
$$
is injective. The multiplicity matrix of the map above, which is given by the  submatrix
\begin{equation*} 
 \begin{pmatrix}I_{r_n} & 0\\A_n^{(2,1)} &A_n^{(2,2)}  \end{pmatrix},\end{equation*}
of $A_n$, cf.\ \eqref{eq:E_n}, must therefore have non-zero columns. Hence $A_n^{(2,2)}$ has non-zero columns.

Finally, $u(n+1,j) = k_j = u(n,j)$,  for each $n \ge 1$ and each $1 \le j \le r_n$. 
\end{proof}

\noindent Every AF-algebra arises from a sequence of simplicial groups as in \eqref{eq:dl}, and it is well-known that the AF-algebra is simple if each entry of each of the matrices $A_n$, $n \ge 1$, is non-zero. With this in mind, we can (loosely) interpret Definition~\ref{def:RFD-JI}  and the theorem below as saying that the class of RFD just-infinite AF-algebras inside the class of all RFD AF-algebras is similar to the class of simple AF-algebras inside the class of all AF-algebras.

\begin{theorem} \label{thm:ji-Bratteli}
Any AF-algebra $\cA$ associated with a sequence of simplicial groups \eqref{eq:dl} with property (RFD-JI) is RFD and just-infinite. Moreover, 
$$\mathrm{Prim}(\cA) = \{0, \mathrm{ker}(\pi_1), \mathrm{ker}(\pi_2), \mathrm{ker}(\pi_3), \dots\},$$
where $\pi_j$, $j \ge 1$, are the irreducible representations determined in Lemma~\ref{lm:RFD-maps}.
\end{theorem}

\begin{proof} Let $\cA$ be an AF-algebra obtained from a  sequence of simplicial groups as in \eqref{eq:dl} with property (RFD-JI). Write $\cA$ as the inductive limit of the sequence
$$\xymatrix{\cA_1 \ar[r]^-{\varphi_1} & \cA_2 \ar[r]^-{\varphi_2} &\cA_3 \ar[r]^-{\varphi_3} & \cdots \ar[r] & \cA},$$
where $\cA_n = \bigoplus_{i=1}^{m_n} \cA_{n,i}$,  $\cA_{n,i} = M_{u_n(i)}$, and where $u_n(i) = k_i$, when $1 \le i \le r_n$. 
Then $\cA$ is RFD by Theorem~\ref{thm:Bratteli-rfd}. We must show that $\cA$ also is just-infinite. For this purpose, let $I$ be a non-zero closed two-sided ideal of $\cA$, and set $I_j = \varphi_{\infty,j}^{-1}(I) \subseteq \cA_j$, for all $j \ge 1$. Then $I_j = \bigoplus_{k \in T_j} \cA_{j,k}$ for some subset $T_j$ of $\{1,2, \dots, m_j\}$. The quotient $\cA/I$ is isomorphic to the inductive limit of the sequence
\begin{equation} \label{eq:quotient}
\xymatrix{\bigoplus_{k \in F_1} \cA_{1,k} \ar[r]^-{\varphi'_1} & \bigoplus_{k \in F_2} \cA_{2,k} \ar[r]^-{\varphi'_2} & \bigoplus_{k \in F_3} \cA_{3,k} \ar[r]^-{\varphi'_3} & \cdots,}
\end{equation}
where $F_j = \{1,2, \dots, m_j\}\setminus T_j$, and where $\varphi'_j$ is the restriction of $\varphi_j$ to $\bigoplus_{k \in F_j} \cA_{j,k}$. 

Choose $j_0 \ge 0$ such that $I_{j_0} \ne 0$ (as we can by standard properties of inductive limit \Cs s).  Set $F = \{1,2, \dots, r_{j_0}\} \setminus T_{j_0}$. Then 
\begin{equation} \label{eq:Tj}
T_j = \{1,2,\dots, m_j\} \setminus F, 
\end{equation}
for all $j > j_0$. To see this, we use the following general facts about ideals: Let $j \ge 1$. If $k \in T_j$ and $A_j(\ell,k) \ne 0$, then $\ell \in T_{j+1}$. Moreover, if $\ell \in T_{j+1}$, for all $\ell$ for which $A_j(\ell,k) \ne 0$, then $k \in T_j$. 
Property (RFD-JI) implies that $A_j(\ell,k) \ne 0$, for all $1 \le k \le m_j$ and $r_{j} < \ell \le m_{j+1}$, and that $A_j(\ell,k) = \delta_{k,\ell}$, when $1 \le k \le m_j$ and $1 \le \ell \le r_j$.
With this information, we can conclude that $\{1,2, \dots, m_j\} \setminus F \subseteq T_j$, for all $j > j_0$. Conversely, if $j \ge j_0$ and $k \in \{1,2, \dots, r_j\}$, then $A_j(\ell,k) \ne 0$ if and only if $\ell \in \{k, r_j+1, r_j+2, \dots, m_{j+1}\}$, and since $\{r_j+1, r_j+2, \dots, m_{j+1}\} \subseteq T_{j+1}$, we conclude that $k \in T_j$ if and only if $k \in T_{j+1}$.  This proves that  \eqref{eq:Tj} holds, for all $j > j_0$.

Hence $F_j = F$, for all $j  > j_0$. By the properties of the multiplicity matrix $A_j$, we see that 
$$\varphi'_j \colon \bigoplus_{k \in F} \cA_{j,k} \to  \bigoplus_{k \in F} \cA_{j+1,k}$$
is an isomorphism, whenever $j > j_0$. It therefore follows from \eqref{eq:quotient} that $\cA/I$ is isomorphic to the finite dimensional \Cs{} $\bigoplus_{k \in F} \cA_{j_0+1,k}$. This proves that $\cA$ is just-infinite.

To prove the last claim, recall first that $0$ is a prime ideal  in any just-infinite \Cs, cf.\ \cite[Lemma 3.2]{GMR:JI}, (and hence primitive, since $\cA$ is separable). The representations $\pi_j$, $j \ge 1$, are irreducible, so their kernels are primitive ideals. Conversely, suppose that $I$ is a non-zero primitive ideal of $\cA$. Then $\cA/I$ is isomorphic to $\bigoplus_{k \in F} \cA_{j,k}$, for some $j \ge 1$ and some finite subset $F$ of $\{1,2, \dots, r_j\}$, by the argument above. As $A/\cI$ is primitive, and hence prime, $F$ must be a singleton, viz.\ $F= \{j\}$, for some $j \ge 1$. However, in that case $I = \mathrm{ker}(\pi_j)$, which proves that $\mathrm{Prim}(\cA)$ is as claimed.
\end{proof}

\noindent We end this paper by showing that not all RFD AF-algebras arise in the way described in Section~\ref{sec:simplex}, i.e., from a Bratteli diagram of the form given in \eqref{eq:Bratteli}, nor, more generally, as in \eqref{eq:dl}, \eqref{eq:algebras}, and \eqref{eq:E_n} with $r_n = m_n$, for all $n \ge 0$. Hence, in general, we cannot collapse the multiplicity matrix $A_n$ in \eqref{eq:E_n} to a matrix of the form
\begin{equation} \label{eq:A_n}
A_n = \begin{pmatrix}I_{r_n} \\B_n \end{pmatrix},
\end{equation}
for some $(r_{n+1}-r_n) \times r_n$ matrix $B_n$ over $\Z^+$. 

An ideal in a \Cs{} is said to be \emph{compact} if it corresponds to a compact subset of the primitive ideal space of the \Cs. All ideals generated by a finite number of projections are compact. 

\begin{lemma} \label{lm:not-A}
Let $\cA$ be a unital RFD AF-algebra which is realized from a sequence of simplicial groups with property (RFD), where $m_n=r_n$, for all $n \ge 1$, i.e., where the multiplicity matrices $A_n$ are of the from \eqref{eq:A_n} above. Then the quotient of $\cA$ by any proper compact ideal $I$ of $\cA$ has a (non-zero) finite dimensional representation. In particular, $\cA/I$ cannot be simple and infinite dimensional.
\end{lemma}

\begin{proof} Write $\cA = \varinjlim(\cA_n,\varphi_n)$, with  $\cA_n = \bigoplus_{i=1}^{r_n} \cA_{n,i}$, where $\cA_{n,i} = M_{k_i}$, and where the connecting maps $\varphi_n$ are unital with multiplicity matrix $A_n$ in the form \eqref{eq:A_n}.  

In the notation of the proof of Theorem~\ref{thm:ji-Bratteli}, set $I_n = \varphi_{\infty,n}^{-1}(I) = \bigoplus_{j \in T_n} \cA_{n,j}$, where $T_n$ is some subset of $\{1,2, \dots, r_n\}$,  and set $F_n = \{1,2, \dots, r_n\} \setminus T_{n}$, for all $n \ge 1$. As in the proof of Theorem~\ref{thm:ji-Bratteli}, $\cA/I$ is the inductive limit of the sequence
\begin{equation} \label{eq:quotient2}
\xymatrix{\bigoplus_{j \in F_1} \cA_{1,j} \ar[r]^-{\varphi'_1} & \bigoplus_{j \in F_2} \cA_{2,j} \ar[r]^-{\varphi'_2} & \bigoplus_{j \in F_3} \cA_{3,j} \ar[r]^-{\varphi'_3} & \cdots \ar[r] & \cA/I,}
\end{equation}
where $\varphi'_n$ is the restriction of $\varphi_n$ to $\bigoplus_{j \in F_n} \cA_{n,j}$. 

By compactness, $I$ is generated, as a closed two-sided ideal in $\cA$, by $\varphi_{\infty,n_0}(I_{n_0})$, for some $n_0 \ge 1$. We may assume that $n_0=1$. As $I$ is a proper ideal and the inclusion $\cA_1 \to \cA$ is unital, $F_1$ is non-empty. Choose $\ell \in F_1$. As $A_n(\ell,j)$ is non-zero if and only if $j=\ell$, by \eqref{eq:A_n}, it follows that $\ell \in F_n$, for all $n \ge 1$. Now, $\cA_{n,\ell} = M_{k_\ell}$, so arguing as in the proof of Lemma~\ref{lm:RFD-maps}, we obtain (surjective) \sh s $\pi_\ell^{(n)} \colon \bigoplus_{j \in F_n} \cA_{n,j} \to M_{k_\ell}$ satisfying 
$\pi_\ell^{(n+1)}  \circ \varphi'_n = \pi_\ell^{(n)}$, 
for all $n \ge 1$. This, in turns, induces a (non-zero) \sh{} $\pi'_\ell \colon \cA/I \to M_{k_\ell}$, as desired.
\end{proof}

\begin{example} \label{ex:CAR-quotient} 
We give here an example of a RFD AF-algebra $\cA$ which admits a  compact ideal $I$ such that $\cA/I$ is simple and infinite dimensional. By Lemma~\ref{lm:not-A} this implies that $\cA$ cannot be realized from a (RFD) sequence of simplicial groups, where $m_n=r_n$, for all $n \ge 1$, cf.\ Definition~\ref{def:RFD-JI}, or with $A_n$ given as in \eqref{eq:A_n}, or from the construction described in Section~\ref{sec:simplex}.

Consider the following Bratteli diagram, written in two different ways:
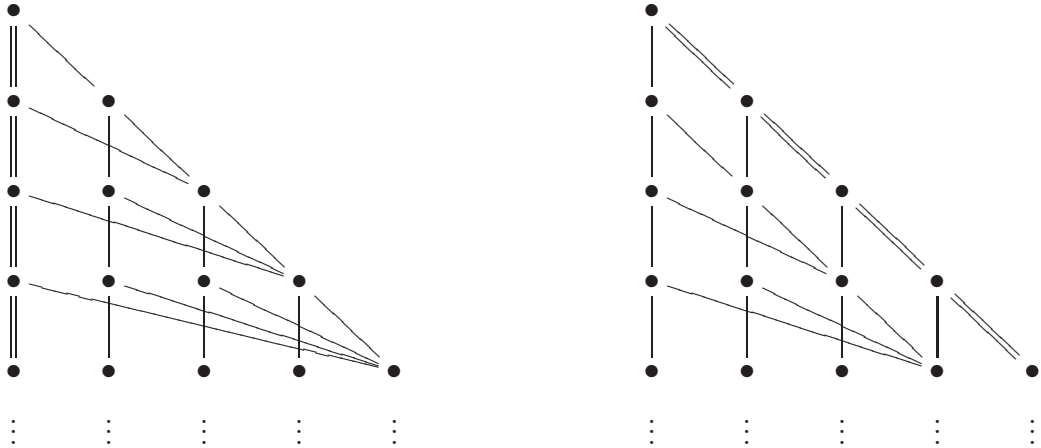
\begin{center}
\begin{equation*} 
\xymatrix@R-1.3pc{\bullet \ar@{=}[dd] \ar@{-}[ddr] & & & & \\ & & & &\\
\bullet \ar@{=}[dd] \ar@{-}[ddrr] & \bullet  \ar@{-}[dd] \ar@{-}[ddr] && &\\  & & & &\\
\bullet  \ar@{=}[dd] \ar@{-}[ddrrr]& \bullet  \ar@{-}[dd] \ar@{-}[ddrr]& \bullet  \ar@{-}[dd] \ar@{-}[ddr]& & \\  & & & &\\
\bullet  \ar@{=}[dd] \ar@{-}[ddrrrr] & \bullet  \ar@{-}[dd] \ar@{-}[ddrrr] & \bullet  \ar@{-}[dd] \ar@{-}[ddrr]& \bullet  \ar@{-}[dd] \ar@{-}[ddr]& \\  & & & &\\
\bullet  & \bullet   & \bullet   & \bullet  & \bullet  \\  \vdots & \vdots & \vdots & \vdots & \vdots
} 
\hspace{3cm}
\xymatrix@R-1.3pc{\bullet \ar@{-}[dd] \ar@{=}[ddr] & & & & \\ & & & &\\
\bullet \ar@{-}[dd] \ar@{-}[ddr] & \bullet  \ar@{-}[dd] \ar@{=}[ddr] && &\\  & & & &\\
\bullet  \ar@{-}[dd] \ar@{-}[ddrr]& \bullet  \ar@{-}[dd] \ar@{-}[ddr]& \bullet  \ar@{-}[dd] \ar@{=}[ddr]& & \\  & & & &\\
\bullet  \ar@{-}[dd] \ar@{-}[ddrrr] & \bullet  \ar@{-}[dd] \ar@{-}[ddrr] & \bullet  \ar@{-}[dd] \ar@{-}[ddr]& \bullet  \ar@{-}[dd] \ar@{=}[ddr]& \\  & & & &\\
\bullet  & \bullet   & \bullet   & \bullet  & \bullet  \\  \vdots & \vdots & \vdots & \vdots & \vdots
} 
\end{equation*}
\end{center}
and let $\cA$ be the unital AF-algebra arising from this diagram (where we assume the connecting mappings all are unital and that the full matrix algebra corresponding to the vertex in the first row is $\C$). Arguing as in Lemma~\ref{lm:RFD-maps}, from the point of view of the left-hand diagram, we have surjective \sh s 
$$\pi_1 \colon \cA \to M_{2^\infty}, \qquad \pi_j \colon \cA \to M_{k_j}, \quad j = 2,3, \dots,$$
for suitable integers $k_j$, $j \ge 2$, and where $M_{2^\infty}$ denotes the CAR-algebra (aka the UHF-algebra of type $2^\infty$). As $\pi_1$ is not injective (because $\cA$ cannot be simple), it follows in particular that $\cA$ is not just-infinite. The Bratteli diagram on the right-hand side has property (RFD), cf.\ Definition~\ref{def:RFD-JI} (when interpreted as a sequence of simplicial groups), with $r_n = n$, $m_n = n+1$, and with
$$A_n = \left( 
\begin{array}{c|c}
I_{n-1} & 0 \\
\hline  1 \cdots 1 & 1  \\
\hline 0 \cdots 0 & 2
\end{array}
\right).
$$
Since $A_n^{(2,2)} = (1)$ is a non-zero $1 \times 1$ matrix, Theorem~\ref{thm:Bratteli-rfd} yields that $\cA$ is RFD. One can also directly verify that $\bigoplus_{j=2}^\infty \pi_j$ is injective. As $\mathrm{ker}(\pi_1)$ is generated by a single projection (eg., a non-zero projection in the summand corresponding to the vertex at position $(2,1)$ in the right-hand side Bratteli diagram), it is compact, so it follows from Lemma~\ref{lm:not-A} that $\cA$ has the stated properties.

Let us finally note that the AF-algebra $\cB$ given by the Bratteli diagam
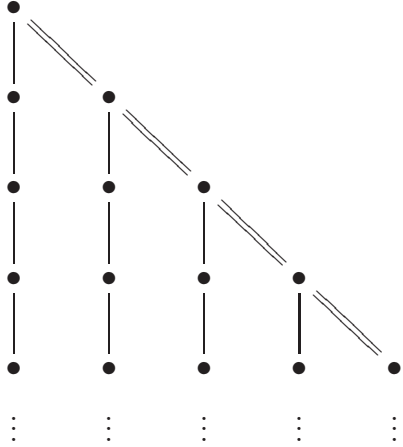
\begin{center}
\begin{equation*} 
\xymatrix@R-1.3pc{\bullet \ar@{-}[dd] \ar@{=}[ddr] & & & & \\ & & & &\\
\bullet \ar@{-}[dd]  & \bullet  \ar@{-}[dd] \ar@{=}[ddr] && &\\  & & & &\\
\bullet  \ar@{-}[dd] & \bullet  \ar@{-}[dd] & \bullet  \ar@{-}[dd] \ar@{=}[ddr]& & \\  & & & &\\
\bullet  \ar@{-}[dd]  & \bullet  \ar@{-}[dd]  & \bullet  \ar@{-}[dd] & \bullet  \ar@{-}[dd] \ar@{=}[ddr]& \\  & & & &\\
\bullet  & \bullet   & \bullet   & \bullet  & \bullet  \\  \vdots & \vdots & \vdots & \vdots & \vdots
} 
\end{equation*}
\end{center}
also has a closed two-sided ideal $I$ with $\cB/I$ being isomorphic to the CAR-algebra $M_{2^\infty}$. It is of the form given in \eqref{eq:dl}, \eqref{eq:algebras} with
$$A_n = \left(  \begin{array}{c}
I_{n} \\
\hline  0 \cdots 0  \; 2 
\end{array}
\right),$$
i.e, with $r_n = m_n = n+1$ and $A_n$ as in \eqref{eq:A_n}. In particular, $\cB$ is RFD (but not just-infinite). This does not conflict with  Lemma~\ref{lm:not-A} because the ideal $I$ is not compact.

The \Cs{} $\cB$ admits the following concrete description as an extension:
$$
\xymatrix{ 0 \ar[r] & \bigoplus_{n=1}^\infty M_{2^n} \ar[r]  & \prod_{n=1}^\infty M_{2^n} \ar[r] & \frac{\prod_{n=1}^\infty M_{2^n}}{\bigoplus_{n=1}^\infty M_{2^n}}   \ar[r] & 0 
\\  0 \ar[r] & \bigoplus_{n=1}^\infty M_{2^n} \ar[r]  \ar@{=}[u] & \cB \ar@{^{(}->}[u] \ar[r] & M_{2^\infty} \ar[u]   \ar[r] & 0
}
$$

\end{example}

{\small{
\bibliographystyle{amsplain}

\providecommand{\bysame}{\leavevmode\hbox to3em{\hrulefill}\thinspace}
\providecommand{\MR}{\relax\ifhmode\unskip\space\fi MR }
\providecommand{\MRhref}[2]{%
  \href{http://www.ams.org/mathscinet-getitem?mr=#1}{#2}
}
\providecommand{\href}[2]{#2}

}}

\vspace{1cm}

\noindent Mikael R\o rdam \\
Department of Mathematical Sciences\\
University of Copenhagen\\ 
Universitetsparken 5, DK-2100, Copenhagen \O\\
Denmark \\
rordam@math.ku.dk\\

\end{document}